\documentclass[12pt]{article}
\usepackage{amsthm}
\usepackage{amsmath}
\usepackage{amssymb}
\usepackage[margin=1in]{geometry}
\usepackage{hyperref}
\usepackage{xcolor}
\usepackage{tikz}
\hypersetup{
    colorlinks,
    linkcolor={black},
    citecolor={black},
    urlcolor={black}
}

\usetikzlibrary{matrix}

\newtheorem{theo}{Theorem}[section]

\newtheorem{lemma}[theo]{Lemma}

\newtheorem{defi}[theo]{Definition}

\newcommand{\ZZ}{\mathbf{Z}} 
\newcommand{\NN}{\mathbb{N}} 
 
\newcommand{\FF}{{\mathcal{F}}}  
\newcommand{\GG}{{\cal G}}

\title{\huge On the Hat Guessing Number of Graphs}
\author{Noga Alon
\thanks
{Department of Mathematics, Princeton University,
Princeton, NJ 08544, and
Schools of Mathematics and
Computer Science, Tel Aviv University, Tel Aviv 69978,
Israel.
Email: {\tt nalon@math.princeton.edu}.
Research supported in part by
NSF grant DMS-1855464 and BSF grant 2018267.}
\and
Jeremy Chizewer
\thanks
{Department of Computer Science, Princeton University,
Princeton, NJ 08544.
Email: {\tt chizewer@princeton.edu}.}
}
\begin{document}
\date{}
\maketitle
\begin{abstract}
The hat guessing number $HG(G)$ of a graph $G$ on $n$ vertices is defined
in terms of the following game: $n$ players are placed on the $n$ vertices
of $G$, each wearing a hat whose color is arbitrarily chosen from a
set of $q$ possible colors.  Each player can see the hat colors of his
neighbors, but not his own hat color.  All of the players are asked to
guess their own hat colors simultaneously, according to a predetermined
guessing strategy and the hat colors they see, where no communication
between them is allowed.  The hat guessing number $HG(G)$ is the largest
integer $q$ such that there exists a guessing strategy guaranteeing at
least one correct guess for any hat assignment of $q$ possible colors.

In this note we construct a planar graph $G$ satisfying $HG(G)=12$,
settling a problem raised in \cite{BDFGM}. We also improve the known
lower bound of $(2-o(1))\log_2 n$ 
for the typical hat guessing  number of the random graph
$G=G(n,1/2)$, showing that it is at least $n^{1-o(1)}$ with probability
tending to $1$ as $n$ tends to infinity. Finally, we consider the linear
hat guessing number of complete multipartite graphs.
\end{abstract}
\section{Introduction}

The following hat guessing game was introduced in \cite{BHKL}.  Let $G$
be a simple graph on $n$ vertices $\{v_1,\ldots,v_n\}$, and let $Q$ be a
finite set of $q$ colors.  The $n$ vertices of the graph are identified
with $n$ players, where each is assigned arbitrarily a hat
colored with one of the colors in $Q$.  A player can only see the hat
colors of his neighbors, i.e., player $i$ sees the hat color of player
$j$ if and only if $v_i$ is connected to $v_j$ in $G$.  After all the
players agree on a guessing strategy, they are asked to guess their own
hat colors simultaneously, with no communication.  The goal of the players
is to ensure that at least one player guesses his hat color correctly.
Let $HG(G)$ denote the maximum number $q$ for which 
there exists a winning guessing strategy for the players.

If $q$ is a prime power, the set $Q$ is identified with the set of
elements of the finite field $F=GF(q)$, and the guessing functions 
are linear (or affine), we call the guessing strategy a linear strategy.

The invariant $HG(G)$ has been studied in several papers including
\cite{BHKL,GG,KL,Szc,Gad,ABST,BDFGM,HL,HIP}. In \cite{BDFGM} the 
authors conjecture that the maximum possible hat guessing number 
of a planar graph is $4$.
The following result shows that it is significantly larger.
\begin{theo} 
\label{planar1}
There exists a planar graph with hat guessing number 12.
\end{theo}

We have established this result in September 2020 (as
mentioned in \cite{HIP}). At that time this has been 
the planar graph with the largest hat guessing number 
known. Examples of planar graphs with hat guessing number $6$
appear in \cite{HL} and in \cite{HIP}. A more recent paper 
\cite{LK} contains a construction of a planar graph with
hat guessing number $14$.

Another conjecture suggested in \cite{BDFGM} is that the 
hat guessing number of any graph 
is at most its Hadwiger's number, that is, the
order of the largest clique minor of the graph.
The theorem above provides, of course, a counterexample to 
that as well. For larger Hadwiger numbers $d$ the hat guessing number 
can in fact be at least 
doubly-exponential in $d$. This follows from the results in \cite{HL}, 
as it is easy to show that the 
Hadwiger number of the graphs $G_d(N)$ constructed there and 
discussed in Theorem 1.1 
in that paper is $d+1$. Similarly (though less dramatically) the 
book graph $B_{d,m}$ for $m>m_0(d)$ has 
Hadwiger number $d + 1$ and hat guessing number exceeding $d^d$.

It is interesting to determine or estimate the hat guessing number 
of the random graph 
$G = G(n, 1/2).$ In \cite{BDFGM} it is shown that with high probability,
that is, with probability tending to $1$ as $n$ tends to infinity, 
$(2 - o(1)) \log_2 n \leq HG(G) \leq {n - (1 + o(1)) \log_2 n}$. 
The following result improves the 
lower bound considerably.
\begin{theo}
\label{random}
Let $G = G(n, 1/2)$ denote the binomial random graph. Then with 
high probability, that is, 
with probability tending to $1$ as $n$ tends to infinity, 
$HG(G) \geq n^{1 - o(1)}$.
\end{theo}

Our final result in this note deals with linear guessing
strategies.  Here $q$ is a prime power, the colors are identified
with the elements of the finite field $F=GF(q)$, and 
each vertex guesses according to
an affine function of the colors of its neighbors.
\begin{defi}[\cite{ABST}]
The linear hat guessing number of $G$, denoted $HG_{lin}(G)$, 
is the largest prime power
$q$ for which there is a winning linear guessing strategy for $G$
over the field 
$GF(q)$
\end{defi}
The linear hat guessing number of $K_{n,n}$, the complete 
bipartite graph with $n$ vertices
in each color class, is studied in \cite{ABST}. Here we consider 
the complete $m$-partite graph,
$K^{(m)}_n$, with $n$ vertices in each color class.
\begin{theo}
\label{linear}
There exists an absolute constant $c$ so that the following holds.
Let $F$ be a field with at least $(c \log (nm))^m$ elements. Then 
there is no linear hat guessing scheme for $K_n^{(m)}$ over $F$.
\end{theo}

The above results are proved in the next three sections. The final
brief section contains several open problems.

\section{Proof of Theorem~\ref{planar1}}

\begin{proof}[Proof of Theorem~\ref{planar1}]
Let $G$ be the planar graph consisting of two adjacent vertices $u,v$
together with $m = 66^{144}$ pairs of adjacent vertices $x_i, y_i$,
each adjacent to both $u$ and $v$. See Figure \ref{planar}, and note
that we make no attempt to minimize $m$, which can be easily reduced.
Call $uv$ the central pair of $G$, and all edges $x_iy_i$ the outer pairs.
Let $\mathcal{F}$ denote the family of all $m = 66^{144}$ functions $f$
from ordered pairs of elements of the cyclic group $\ZZ_{12}$ to unordered
pairs of elements of $\ZZ_{12}$. Associate each outer pair with one
such function according to some fixed bijection between these functions
and the outer pairs.  The guessing strategies of the vertices in the
outer pairs are defined as follows. Consider such a pair $xy$, let $f$
be the function corresponding to it, and let $h(u), h(v), h(x), h(y)$ be
the hat colors of $u, v, x$ and $y$, respectively, viewed as elements of
$\ZZ_{12}$. The value of $f(h(u),h(v))$ is an unordered pair of elements
$\{g_1,g_2\}$ of $\ZZ_{12}$ with, say, $g_1 < g_2$. Let $x$ guess that
its hat color is $g_1 - h(y)$ (computed in $\ZZ_{12}$) and let $y$ guess
that its hat color is $g_2 - h(x)$ (computed in $\ZZ_{12}$.) Note that
both $x, y$ guess incorrectly if and only if the sum of hat colors of
$x,y$ is not one of the two elements in $f(h(u),h(v))$. We claim that
for any fixed hat colors of all vertices in the outer pairs there are at
most $5$ distinct colorings of the hats of $u, v$ in which all vertices
in the outer pairs guess incorrectly. Indeed, if there are at least $6$
such colorings then there is some function $f \in \mathcal{F}$ (in fact
many such functions) that maps these $6$ colorings to pairwise disjoint
pairs of elements of $\ZZ_{12}$ whose union covers all of $\ZZ_{12}$.
Let $x,y$ be an outer pair corresponding to this function.  Then, since
$x, y$ fail to guess correctly in all these $6$ colorings, it follows that
the sum of the hat colors of $x,y$ cannot be any element of $\ZZ_{12}$,
a contradiction. This proves the claim and shows that in case all the
vertices in the outer pairs guess incorrectly there are only $5$ possible
distinct colorings of the pair $u, v$.  The vertices $u, v$ see all hats
of the vertices in the outer pairs and hence can compute this set of
at most $5$ distinct colorings of their hats. They can then take care
of these colorings ensuring that at least one of them guesses correctly
since as shown in \cite{BDFGM, HL} a clique of size $2$ can handle any
(known) set of $5$ colorings.  This shows that the hat guessing number
of our graph is at least $12$. We next show that it is smaller than $13$
(and hence it is exactly $12$).

Consider the game on this graph with $13$ colors. We will pre-commit
to using a  member of a fixed set of $6$ colorings for the vertices
of the central pair. For each coloring of these  vertices, each
vertex in an outer pair can guess correctly in at most $13$ possible
colorings of the outer pair, one for each color of the other vertex
in the pair. Hence each outer pair can handle at most $26$ different
colorings of the pair. Therefore for all $6$ colorings of the central
pair, this outer pair can have at least one vertex guessing correctly
in at most $26\cdot 6= 156$ possible colorings of the pair. Since there
are $13^2=169$ possible colorings of an outer pair, there is at least
one (in fact, at least $169-156=13$) 
coloring where they guess incorrectly for all $6$ colorings of the
central pair. For each outer pair we can choose a coloring in which both
its vertices guess incorrectly for all $6$ colorings.  This shows
that for any fixed set of $6$ possible colorings of the central pair,
there are choices for the colors of all outer pairs (for any number of
outer pairs) in which all vertices of the outer pairs guess incorrectly.

As shown  in \cite{BDFGM, HL} there
exists a set of $6$ colorings, for example the set $\{0,1\}\times\{0,1,2\}$,
such that the central vertices cannot guess correctly on all of
them. Suppose the central pair gets one of these $6$ colorings, and every
outer pair gets a coloring in which both its members  guess
incorrectly for any of these $6$ colorings. Each of the two vertices in the
central pair sees the colors of all other vertices, and knows the
guessing function of each vertex. Crucially, even with this
information, each of these two vertices knows that the central pair
may well have any of the above $6$ colorings. As for any strategy
of the central pair, both its members guess incorrectly on at least
one of these $6$ colorings, this shows
that there is no
winning strategy for the graph with $13$ colors. Therefore the hat guessing
number of this graph is exactly $12$.
\end{proof}

\begin{figure}
\begin{center}
\begin{tikzpicture}
  \matrix (m) [matrix of math nodes,row sep=8em,column sep=2em, nodes={circle,draw,fill} ]
  {
     . & & & & & & & &  \\
        & . & . &  . & . &  &  & . & . \\
     . & & & & & & & &\\ };
  \path[]
	(m-1-1) 	edge (m-3-1)
			edge (m-2-2)
			edge (m-2-3)
			edge (m-2-4)
			edge (m-2-5)
			edge (m-2-8)
			edge (m-2-9)
	(m-3-1)	edge (m-2-2)
			edge (m-2-3)
			edge (m-2-4)
			edge (m-2-5)
			edge (m-2-8)
			edge (m-2-9);
   \path[thick]
	(m-2-2)	edge (m-2-3)
	(m-2-4)	edge (m-2-5)
	(m-2-8)	edge (m-2-9);
	
   \path[ultra thick,dash pattern=on 0 cm off 0.5cm on .15cm off .15cm on .15cm off .15cm on .15cm off 
   .15cm on .15cm off .15cm on .15cm off 1 cm]
   	(m-2-5) 	edge (m-2-8);

\end{tikzpicture}\\
\end{center}
\caption{A planar graph with hat guessing number 12.}
\label{planar}
\end{figure}
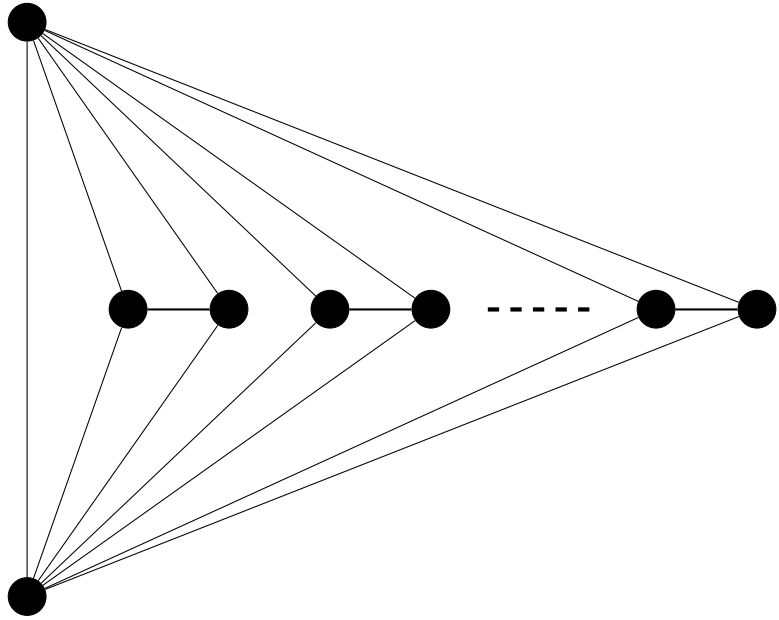
\section{Proof of Theorem~\ref{random}}

Our lower bound on the hat guessing number of the 
random graph uses the fact that the random graph contains a 
a sufficiently
large book graph as a subgraph with high probability.

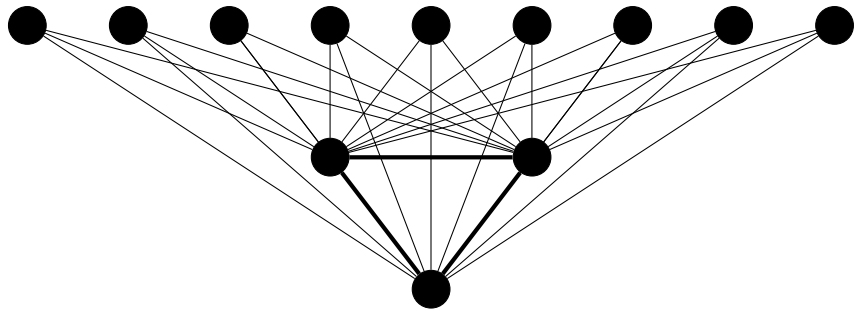
\begin{figure}
\begin{center}
\begin{tikzpicture}
  \matrix (m) [matrix of math nodes,row sep=3em,column sep=2em, nodes={circle,draw,fill} ]
  {
     . & . & . & . & . & . & . & . & .  \\
       &   &   &  . &   & .  &   &   &  \\
       &   &   &    & . &    &   &   & \\ };
  \path[]
	(m-1-1) 	edge (m-2-4)
			edge (m-2-6)
			edge (m-3-5)
	(m-1-2) 	edge (m-2-4)
			edge (m-2-6)
			edge (m-3-5)
	(m-1-3) 	edge (m-2-4)
			edge (m-2-6)
			edge (m-3-5)
	(m-1-4) 	edge (m-2-4)
			edge (m-2-6)
			edge (m-3-5)
	(m-1-5) 	edge (m-2-4)
			edge (m-2-6)
			edge (m-3-5)
	(m-1-6) 	edge (m-2-4)
			edge (m-2-6)
			edge (m-3-5)
	(m-1-7) 	edge (m-2-4)
			edge (m-2-6)
			edge (m-3-5)
	(m-1-8) 	edge (m-2-4)
			edge (m-2-6)
			edge (m-3-5)
	(m-1-9) 	edge (m-2-4)
			edge (m-2-6)
			edge (m-3-5);	
  \path[ultra thick]
	(m-2-4)	edge (m-2-6)
			edge (m-3-5)
	(m-2-6)	edge (m-3-5);

\end{tikzpicture}\\
\end{center}
\caption{The book graph $B_{3,9}$}
\label{book}
\end{figure}

The book graph 
denoted $B_{d,m}$ is the graph obtained by taking a complete graph
on $d$ vertices, which we call the central clique, adding
to it an independent set of size $m$, and joining each of its
vertices to all the
vertices of the central clique. Figure \ref{book}
shows a drawing of $B_{3,9}$. The hat guessing numbers of book 
graphs are studied in \cite{BDFGM}, \cite{HL}, \cite{HIP}.
The following result is proved in \cite{HIP}.
\begin{theo} [\cite{HIP}]
For every fixed $d$ and $m$ sufficiently large as a function of
$d$, 
$$HG(B_{d,m}) = \sum_{i=1}^d i^i + 1.$$
\end{theo}
We cannot use this result to prove Theorem \ref{random}, since
the assertion above holds only for 
$m$ which is very large as a function of $d$. For our purpose  
we need to show that the hat guessing number of $B_{d,m}$ is close
to the expression above 
even if $m$ is much smaller. This is proved in the following
lemma.  It is worth noting that
both the value of $m$ and that of $q$ can be 
slightly improved in this lemma, but the estimate below suffices 
for our purpose.
\begin{lemma}  For $m = d^d \cdot d^3, HG(B_{d,m}) \geq q = d^d/d^2$.
\label{lemma}
\end{lemma}
\begin{proof}
Let $q = d^d/d^2 = d^{d - 2}$. Call the 
clique of $B$ the central clique, and the other vertices the 
outer ones. We claim that there are $m$ 
functions $f_1, \dots , f_m$ from $\ZZ_q^d$, the set of possible 
colorings of the central clique to $\ZZ_q$ 
so that for every subset $S\subset \ZZ_q^d$ of size $|S|=d^d$ 
at least one of the functions restricted 
to $S$, $f_i|_S$ is onto. Indeed, choose these functions randomly 
and uniformly among all possible 
functions from $\ZZ_q^d$ to $\ZZ_q$. For a fixed set, $S$, the 
probability that a fixed function 
misses a specific element of $\ZZ_q$ is $(1-1/q)^{d^d}$. By the union 
bound, the probability that at least one element is missed is at most 
\[q(1-1/q)^{d^d}\leq qe^{-d^d/q} < qe^{-d^2} < 1/2\]
Thus the probability that none of the functions is onto is less
than $(1/2)^m = (1/2)^{d^d\cdot d^3}$. The number of subsets of size
$d^d$ of the set of colorings is \[{q^d\choose d^d} < d^{d^d\cdot
d^2}=2^{d^d\cdot d^2\cdot \log d}\] and the claim follows from the
union bound as $2^{d^d\cdot d^2\cdot \log d}\cdot (1/2)^{d^d\cdot d^3}
< 1$, so with positive probability, for every subset $S\subset \ZZ_q^d$
of size $|S|=d^d$ there is a function $f_i|_S$ which is onto $\ZZ_q$,
proving the existence of a set of $m$ such functions.  Returning to the
proof of the lemma let $f_1, \dots , f_m$ be as in the claim, and let
these be the guessing functions of the outer vertices of $B$. By the
claim it is clear that for any fixed hat colors of the outer vertices
there are less than $d^d$ colorings of the central clique for which all
outer vertices guess incorrectly, since for every set of $d^d$ colorings
at least one of the guessing functions is onto. Thus the central vertices
can compute this known set of less than $d^d$ colorings and as shown in
\cite{BDFGM,HL,HIP} can then handle them.
\end{proof}

\begin{proof}[Proof of Theorem~\ref{random}]
It is well known (c.f., e.g., \cite{AS}) that with high 
probability, the graph ${G = G(n,1/2)}$ contains a clique of size 
$(2 + o(1)) \log_2 n$. Let $d$ be the largest number so 
that $d^d \cdot d^3 \leq 0.5n/2^d$.  
This gives $d = (1+o(1))\log n/\log \log n $. 
Let $K$ be a clique of size $d$ in $G$.  Such a clique exists 
with high probability since $d$ is (much) smaller than $2 \log_2 n$.
The expected number of its
common neighbors is $(n-d)/2^d$ and by the standard estimates for
binomial distributions (c.f., e.g., \cite{AS}, Appendix A),
with high 
probability it has more than $0.5n/2^d \geq d^d \cdot d^3$ 
common neighbors. 
This means that with high 
probability $G$ contains a book $B_{d,m}$ with $m = d^d \cdot d^3$ 
and the result follows from 
Lemma~\ref{lemma}. \end{proof}

\section{Proof of Theorem~\ref{linear}}

The result is proved using an idea in \cite{ALWZ}, where the authors
obtain an improved bound for the sunflower lemma,
and its modifications in \cite{FKNP, Ra, Tao}.
\begin{defi}[\cite{Tao}, Definition 1]
We call a family, $\FF$, 
of sets, $R$-spread if for every nonempty set $Z$
a uniformly random set $F$ chosen from $\FF$ satisfies 
$\mathbb{P}(S\subseteq F)\leq R^{-|Z|}$
\end{defi}

\begin{lemma}[Proposition 5 of \cite{Tao} (following \cite{ALWZ, FKNP, Ra})]
\label{l21}
Let $r,w\geq 2$ be natural numbers, and let $R\geq Cr\log(wr)$ for 
a sufficiently large absolute constant $C$ not 
dependent on $r,w$. Let $\FF=(F_i)_{i\in I}$
be a finite, $R$-spread family of subsets 
of a finite set $T$, each of which having
cardinality at most $w$. Let $V$ be a random subset of $T$ where each
$t\in T$ independently lies in $V$ with probability $1/r$. 
Then with probability greater than $1-1/r$,
$V$ contains an element of $\FF$.
\end{lemma}

\begin{proof}[Proof of Theorem~\ref{linear}]

Denote the vertices of 
$K^{(m)}_n$  by $v_{ij}$, $1 \leq i \leq n, 1 \leq j \leq m$, let
$x_{ij} \in F$ be the hat-value of $v_{ij}$, put $p=|F|$ and let
the linear guessing function of $v_{ik}$ be
\[f_{ik}(x_{ij}: 1 \leq i \leq n, 1 \leq j \leq m, j \neq k)+b_{ik}\]

Put $w=nm$, let $T:=[n]\times[m]\times Z_p$, with $|T|=pnm$.
For each of the $p^{nm}$ possible vectors $x=(x_{ik}) \in Z_p^{nm}$ define
two subsets of size $w=nm$ of $T$ as follows.
\[F(x)=\left\{\Big(i,k, x_{ik}- f_{ik}(x_{ij}: 1 \leq i 
\leq n, 1 \leq j \leq m, j \neq k)-b_{ik}\Big): 
1 \leq i \leq n, 1 \leq k \leq m\right\},\]
\[G(x)=\left\{\Big(i,k,x_{ik}- f_{ik}(x_{ij}: 1 \leq i 
\leq n, 1 \leq j \leq m, j \neq k)\Big):
1 \leq i \leq n, 1 \leq k \leq m\right\}\]
Let $\FF$ be the family of all $p^{nm}$ sets $F(x)$ as above, and let
$\GG$ be the family of all $p^{nm}$ sets $G(x)$.
It is easy to check that each of these two families is $R$-spread
for $R=p^{1/m}$. Indeed, any set containing some given fixed set $Z$ must
satisfy some specific family of $|Z|$ linear equations.
At least $|Z|/m$ of these equations correspond to the same value of $k$,
by the pigeonhole principle, and these 
equations are linearly independent, as
each $x_{ik}$ that appears among them appears only in one of them.
So the number of solutions is at 
most $p^{w-|Z|/m}$, and the total number of 
sets is $p^{w}$. By Lemma~\ref{l21} with $r=2$ there are sets
$F \in \FF$ and $G \in \GG$ which are disjoint.  Indeed, if we
color the elements of $T$ randomly, uniformly and independently 
by $2$ colors
then with positive probability the first color class contains 
a member of $\FF$ and the second a member of $\GG$.

Taking the difference provides a vector of colors
$x=(x_{ij})$ for which none of the differences
$x_{ik}-(f_{ik}(x)+b_{ik})$ vanishes.  This
means that for the corresponding coloring
no vertex guesses correctly, completing the proof. 
\end{proof}

\section{Open Problems}
One of the questions raised in
\cite{ABST} is whether there exists a function 
$f :\NN \to \NN$ such that if 
$G$ is $d$-degenerate, then $HG(G) \leq f(d)$. It seems plausible
that  such a function exists.  If so, then
the maximum possible hat guessing number of any planar graph 
is bounded by an absolute constant, as planar graphs are
$5$ degenerate.  More generally, this would imply that the hat guessing
number of any graph can be upper bounded by a function of its
Hadwiger number, as graphs with Hadwiger number $k$ are known to
be $O(k \sqrt {\log k})$-degenerate, see \cite{Kos}, \cite{Tho}.

Another interesting question is the typical asymptotic behavior of
the hat guessing number of the random graph $G(n,1/2)$. In
particular, is it $o(n)$ with high probability?

The final question we suggest here is that of estimating the
largest possible hat guessing number of a graph $G'$ obtained from
a graph $G$ with hat guessing number $q$ by the addition of a single
vertex (connected to all vertices of $G$). The graphs $G_d(N)$
constructed in \cite{HL} can be used to show that for arbitrarily
large values of $q$ the number can
increase to more than $q^2$. This is obtained by letting 
$G$ be the 
(disconnected) graph obtained from $G'=G_d(N)$ by deleting its
unique vertex  which is connected to all other ones. Thus every connected
component of $G$ is a copy of $G_{d-1}(N)$. As shown in
\cite{HL}, for sufficiently large $N$, if the hat guessing
number of $G$ (which equals that of each of its connected components)
is $q$, then the hat guessing number of $G'$ is
$q^2+q$.

It will be
interesting to establish an upper bound for $HG(G')$ as a function
of $HG(G)$, or to show that no such function exists.
\vspace{0.2cm}

\noindent
{\bf Acknowledgment:}\,
We thank Ben Przybocki for telling us about \cite{LK}.

\end{document}